\documentclass{article}
\usepackage[utf8]{inputenc}
\usepackage{indentfirst}
\usepackage[english]{babel}
\usepackage[doi=false,isbn=false,eprint=false,url=false]{biblatex}
\usepackage[margin=1in]{geometry}
\usepackage[colorinlistoftodos]{todonotes}
\usepackage{amsfonts}
\usepackage{circuitikz}
\usepackage[dvipsnames]{xcolor}
\usepackage{graphicx} 
\usepackage{mathtools, stmaryrd}
\usepackage{amsfonts, amsmath, amssymb, amsthm}

\newtheorem{theorem}{Theorem}
\newtheorem{lemma}{Lemma}
\newtheorem{definition}{Definition}
\newtheorem{corollary}{Corollary}
\newtheorem{remark}{Remark}
\newtheorem{proposition}{Proposition}

\title{Some results on split sum of impartial games}
\author{Carret François}
\date{\today}

\begin{document}

\maketitle

\begin{abstract}
 As it is stated in ``Unsolved problems in combinatorial games" \cite{unsolvedpb2015}  as problem A13, Nim with pass is a difficult problem. Nim was one of the first combinatorial games to be solved but when we introduce a pass the game becomes far more complex. 
 
 In this article we extend the use of the split sum introduced in “Investigations of Impartial Games With a Pass” \cite{hirsch2020} and we find cases where the split sum act the same as an classical disjunctive sum. We also introduce generalization of Grundy value for the split sum and find that games of same Grundy values can give very different properties with a split sum.

 
\end{abstract}

\section{Introduction}

In this paper, we consider two-player \emph{impartial games}, that is, two-player games in which, in every position, the sets of options for both players are the same. We assume that the games are under \emph{normal play convention}, that is, the player who has no legal move \emph{loses}. In a \emph{short game}, a position is
never repeated and there are only a finite number of other positions which can be
reached.

One of the most famous short impartial games is Nim. 
In a game of Nim, there are $n$ piles, each with a finite number of tokens. At their turn a player can take any strictly positive number of tokens from one pile.

In 1901, Bouton in \cite{Bouton1901} solved the game of Nim. Sprague and Grundy proved that any impartial games can be reduced to a Nim position according to disjunctive sum \cite{Sprague1935} \cite{Grundy1939}. The Sprague and Grundy's theorem is useful to compute winning strategy in many impartial games. 

However, for games with a pass, this method seems not to work, and these games seem difficult to analyze. A game with a pass is a game where it is possible to pass a turn but when a player has passed their turn, none of the player can pass anymore. 
Here are some results considering this convention \cite{manabe2025spraguegrundyvaluesgamespass}, \cite{miyadera2022multidimensionalchocolatenimpass}, \cite{miyadera2022restrictednimpass},  \cite{miyadera2017ryuonimvariantclassical}, \cite{Morrison_2011}, \cite{LC15}, \cite{integers1}, \cite{CLLW18}.

In \cite{hirsch2020}, a generalization of game with pass is formalized with the split sum. This generalization can help us find better result on games with pass.

Generally we want to find way to calculate which player has a winning strategy in a game. If we know for each position which player has a winning strategy then we can easily compute winning moves as winning moves are moves to a position where the opponent does not have a winning strategy.


The outline of this paper is as follows. In the latter part of this section, we
introduce necessary previous studies for completeness and introduce split sum. In Section~\ref{result_on_split_sum}, we establish the split value of a game as a counterpart of the Grundy value for split sum of games. We also determine in this section a way to calculate the split value of a game with the split value of its suppositions. In Section~\ref{sec_split_and_nim}, we solve split sum of Nim when we have an important number of times $*$. In Section~\ref{sec_modulo_strong_equality}, we introduce a generalization of Grundy and split values in the objective of characterizing the comportment of a game with split sum. This generalization of Grundy and split values gives us notably that there is a large possibility of behavior for games with a split sum.

\subsection{Impartial Combinatorial Game Theory Definitions and Theorem}

In this section we remind some classic results and definitions from combinatorial game theory and some results that we will use in the other sections. 

For a game $G$ we will note $G=\{G'_1,G'_2,...,G'_n\}$ with $G'_1,G'_2,...,G'_n$ the options of $G$.

One of the main subjects of combinatorial game theory is the sum of different games.

\begin{definition}
The sum of two games G and H is the game where either player on their turn can play on either G or H but not both. We denote by “$+$” the sum and by “$\sum$” the sum of multiple terms.
\end{definition}

\begin{definition}
    We call a game where the previous player has a winning strategy, a $\mathcal{P}$-position or a zero position. We call a game where the next player has a winning strategy, an $\mathcal{N}$-position.
\end{definition}

\begin{definition}
    The birthday of a game is the maximal number of moves that can be played on this game. The birthday also corresponds to the steps for creating the game when we construct games by induction. For a game $G$, we write $b(G)$ the birthday of $G$.
\end{definition}

The Nim game is central in the study of impartial games with the Theorem~\ref{theorem:sprague_grundy}.

\begin{definition}
A Nimber $*n$ is the game where the options are $*k$ for $0 \le k<n$. 

A Nim game is in this article the sum of a certain number of nimbers.

We name $*1$ as $*$ or ``star'' and $*0$ is the empty game also known as $0$.
\end{definition}

It gives us that a Nimber is a game of Nim with a single pile and a Nim game is a game of Nim with any number of piles.

\begin{definition}
Two games $H$ and $G$ are equal if for all $X$, the outcomes of $G+X$ and $H+X$ are the same. We denote equality by $=$. If two games $G$ and $H$ have the same option sets, we denote $G \equiv H$.
\end{definition}

\begin{theorem}
\label{theorem:sprague_grundy}
For each impartial short game $G$ there exists an integer n such that $G=*n$. 
\end{theorem}

This theorem was proved by Sprague and Grundy in \cite{Sprague1935} and \cite{Grundy1939}.

This theorem is central in the study of impartial game as it gives a way to compute the sum of any impartial game starting from a Grundy value that can be computed for every short impartial game.

\begin{definition}
The Grundy value of a game $G$ is the number $n$ such that $G=*n$.
We denote the Grundy value of $G$ by $g(G)$. 
\end{definition}

\begin{definition}
The Nim sum of two non negatives integers $a$ and $b$ is the sum in binary writing without carrying. We denote the Nim sum of $a$ and $b$ by $a \oplus b$.
\end{definition}

\begin{definition}
    The mex function of a set $A$ is the least non negative integer which is not in $A$.
\end{definition}

The Nim sum is named like that because we have that $*a+*b=*(a \oplus b)$. To calculate the Grundy value of a game, we can make the Nim sum of the Grundy values of the different components of the game or take the mex of the set of the Grundy values of its options.

If $G'$ is an option from $G$, we denote $G'\in G$.

\subsection{Split Sum}

For studying impartial games with a pass we introduce a generalization of pass, the split sum. It help us find results on game with pass as well as more complex games that would imply such split sum.

We use the definitions of split sum from \cite{hirsch2020}.
We denote the split sum of $G$ and $H$ by $G \circ H$. In the following definition, $g$ ranges over the options of $G$ and $h$ ranges over the options of $H$.

\begin{definition} 
\label{splitdefinition}
\[G \circ H \equiv
 \begin{cases} 
      0 & \text{if } G \equiv 0, \\
      G & \text{if } H \equiv 0, \\
      
      \{G \circ h, g \circ H\}  & \text{otherwise.}
   \end{cases}
\]
\end{definition}

For practical reason, to represent Nim position we use an isomorphic game on a sequence of cases at each turn one player can move one token to the right as long as it has not reached the right most case represented with a black case. We represent the number of tokens at the left of the split sum in a case with a number in black and the number of tokens in a case at the right of the split sum with a number in green.

\begin{figure}[!ht]
\centering
\resizebox{1\textwidth}{!}{%
\begin{circuitikz}
\tikzstyle{every node}=[font=\fontsize{18.2pt}{23.7pt}\selectfont]
\draw  (0,0) rectangle (1,1);
\draw  (1,0) rectangle (2,1);
\draw  (2,0) rectangle (3,1);
\draw  (3,0) rectangle (4,1);
\draw  (4,0) rectangle (5,1);
\draw  (5,0) rectangle (6,1);
\draw  (6,0) rectangle (7,1);
\draw [ fill={rgb,255:red,0; green,0; blue,0}, fill opacity=1] (7,0) rectangle (8,1);
\node [font=\fontsize{18.2pt}{23.7pt}\selectfont, inner xsep=0.080cm, inner ysep=0.085cm, rounded corners=0.020cm] at (1.25,0.5) {3};
\node [font=\fontsize{18.2pt}{23.7pt}\selectfont, inner xsep=0.080cm, inner ysep=0.085cm, rounded corners=0.020cm] at (2.25,0.5) {5};
\node [font=\fontsize{18.2pt}{23.7pt}\selectfont, color=OliveGreen, text opacity=1, , inner xsep=0.080cm, inner ysep=0.085cm, rounded corners=0.020cm] at (4.75,0.5) {4};
\node [font=\fontsize{18.2pt}{23.7pt}\selectfont, inner xsep=0.080cm, inner ysep=0.085cm, rounded corners=0.020cm] at (4.25,0.5) {3};
\node [font=\fontsize{18.2pt}{23.7pt}\selectfont, color=OliveGreen, text opacity=1, , inner xsep=0.080cm, inner ysep=0.085cm, rounded corners=0.020cm] at (6.75,0.5) {2};
\end{circuitikz}
}%
\caption{Example of the representation of a Nim game $G$}
\label{fig:1}
\end{figure}

For example, Figure~\ref{fig:1} represents the game $(*6+*6+*6+*5+*5+*5+*5+*5+*3+*3+*3)\circ(*3+*3+*3+*3+*+*)$.

The two following theorems come from \cite{hirsch2020} and are useful to understand some properties of the split sum.

\begin{theorem}
Let $G$ be any impartial game and $*n$ be a Nimber.
Let $H$ be any impartial game. Then $G \circ H + *n = 0 \Leftrightarrow (G + *n) \circ  H = 0$.
\end{theorem}

\begin{theorem}
\label{theorem_spit_sum}
Let $G$, $H$, and $K$ be impartial games. Then
$(G \circ H) \circ K = G \circ (H + K).$
\end{theorem}

\begin{remark}
    For a game $G$, $G\circ *$ is the same game as $G$ with a pass as you can consider the move on $*$ as a pass that can be taken only if $G$ is not a terminal position. Therefore split sum is a generalization of games with a pass.
\end{remark}

\section{General Results on Split Sums}
\label{result_on_split_sum}

In this section, we give some interesting results on split sums and introduce the split value as a counterpart of Grundy value for split sums.

\begin{proposition}
\label{righequality}
For all G, H, K, and X short impartial games: 
\[ G \circ (H+X) = G\circ (K+X) \Longleftrightarrow H=K .\]
\end{proposition}
\begin{proof}
We prove this by induction on $b(G)+b(H)+b(K)+b(X)$.

If $H=K$, for $X$ and $G$, second player can win \[G \circ (H+X) + G \circ(K+X)\] by playing the same move as the other player on the other component if possible. If it is not possible to play the same move it means that the move was on $H$ or on $K$. Without loss of generality we assume that the move was on $H$ and the position becomes $G\circ(H'+X)+G\circ(K+X)$. Then, since $H=K$, $K$ has an option $K'$ that satisfies $H'=K'$ or $H'$ has an option $H''=K$. Then by induction hypothesis, we obtain that \[G \circ (H+X) + G \circ(K+X)\] is a win for the second player.

If $H$ is not equal to $K$ then first player can play to have $H'=K$ or $H=K'$ as both are impartial games then the player can win because \[H'=K \implies G \circ (H'+X) = G\circ (K+X) \text{ (and } H=K' \implies G \circ (H+X) = G\circ (K'+X)\text{)}  \] as proved in the first half of the proof. 
\end{proof}

We can then always replace $H$ by $*g(H)$ in $G \circ H$.

On the other hand, if $G=K$ we do not always have $G \circ H = K \circ H$ as for example $* \circ *$ is a first player win but $(*+*+*) \circ *$ is a second player win and $*+*+*=*$.

We then have that either $G$ is empty, either there exists exactly one $n$ such that $G \circ(*n)=0$ or either for all $n$, $G \circ (*n) \ne 0$.

\begin{definition}
For a nonempty game $G$, the \emph{split value} of $G$, $s(G)$ is
\[s(G) =
 \begin{cases} 
      n & \text{ if } G \circ(*n)=0, \\
      \infty &  \text{ if for all $n$, } G \circ (*n) \ne 0.
   \end{cases}
\]

\end{definition}

\begin{remark}
    $s(G)$ is uniquely determined as if $G \circ(*n)=0$ and $G \circ(*m)=0$ then by the previous theorem it implies that $*n=*m$ which means $n=m$.
\end{remark}

Like for Grundy value, we can use a mex rule allowing a better way to compute the split value of a game.

\begin{theorem}
\label{mexrule}
For all nonempty G,
\[s(G) =
 \begin{cases} 
      \infty & \text{ if } 0 \in G \\
      {\rm mex} (\{s(G')\mid G'\in G\})  &  \text{ otherwise.} 
   \end{cases}
\]

\end{theorem}

\begin{proof}
If $0 \in G$ then for all $H$, first player can play from $G \circ H$ to 0 then $G \circ H \ne 0$ thus $s(G)=\infty$. \\

Otherwise, for all $n<{\rm mex}(\{s(G')\})$, first player can play from $G \circ *n$ to a $G' \circ *n$ with $s(G')=n$ then $G \circ *n$ is a $\mathcal{N}$-position. 

For $n={\rm mex}(\{s(G')\})$, from $G \circ *n$ first player can either play on $*n$ or on $G$ but by playing on $*n$ they go to an $\mathcal{N}$-position then $G \circ *n$ is a win for second player and if the first player plays on $G$, by the fact that $n={\rm mex}(\{s(G')\})$, $G' \circ *n \ne 0$. 

Thus, $s(G)={\rm mex}(\{s(G')\})$.
 \end{proof}

Even if the split value of the disjunctive sum of two game is difficult to calculate, there is an easy way to compute the split value of a split sum.

\begin{corollary}
Let G and H be any impartial games then $s(G \circ H) = s(G) \oplus g(H)$, with $ \infty \oplus n=\infty$ for all $n$. 
\end{corollary}

\begin{proof}
Let $G$ and $H$ be any impartial games with $s(G \circ H) \ne \infty$.

We know that $(G \circ H) \circ *s(G \circ H) = 0$.

By Theorem~\ref{theorem_spit_sum} we have $G \circ (*s(G \circ H)+H) = 0$ that implies $*s(G)=*s(G \circ H)+H$ 
which gives us $H=*s(G)+*s(G \circ H)=*(s(G) \oplus s(G \circ H))$.

Then $g(H)=s(G) \oplus s(G \circ H)$ and finally $s(G \circ H) = s(G) \oplus g(H)$.

If $s(G \circ H) = \infty$, it means that there is no $n$ such that:
$(G \circ H) \circ *n = 0$ .

Then by Theorem~\ref{theorem_spit_sum} there is no $n$ such that $G \circ (*n+H) = 0$ but $*n+H$ can be equal to any $*m$ by changing $n$ then by Proposition~\ref{righequality} there is no $m$ such that $G\circ *m =0$ which means that $s(G)=\infty$.
 \end{proof}

\begin{remark}
    For all game $G$, $g(G)=0 \iff s(G)=0$.
\end{remark}

\begin{proof}
    $g(G)=0 \iff G=0 \iff G\circ0=0\iff s(G)=0$.
\end{proof}

\begin{theorem}
\label{theorem:s=1}
    For any impartial game $G$, such as there exist three impartial games $G_1, H_1$ and $H_2$ with $g(H_1)=g(H_2)=1$ and $G\equiv G_1+H_1+H_2$, then $g(G)=1 \iff s(G)=1$.
\end{theorem}

This theorem gives all $\mathcal{P}$-positions and $\mathcal{N}$-positions of game with a pass as long as we have at least two component of Grundy value $1$ because a split value of $1$ signify that it is a $\mathcal{P}$-position in the game with a pass. 

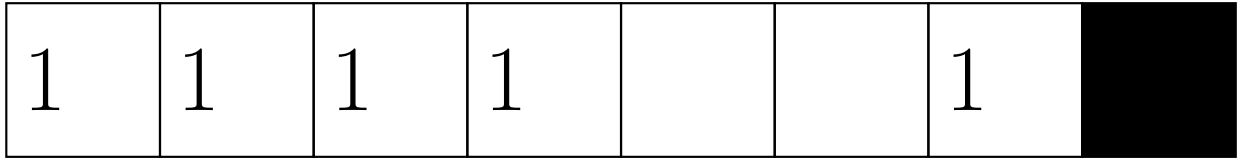
\begin{figure}[!ht]
\centering
\resizebox{1\textwidth}{!}{%
\begin{circuitikz}
\tikzstyle{every node}=[font=\fontsize{18.2pt}{23.7pt}\selectfont]
\draw  (0,0) rectangle (1,1);
\draw  (1,0) rectangle (2,1);
\draw  (2,0) rectangle (3,1);
\draw  (3,0) rectangle (4,1);
\draw  (4,0) rectangle (5,1);
\draw  (5,0) rectangle (6,1);
\draw  (6,0) rectangle (7,1);
\draw [ fill={rgb,255:red,0; green,0; blue,0}, fill opacity=1] (7,0) rectangle (8,1);
\node [font=\fontsize{18.2pt}{23.7pt}\selectfont, inner xsep=0.080cm, inner ysep=0.085cm, rounded corners=0.020cm] at (1.25,0.5) {1};
\node [font=\fontsize{18.2pt}{23.7pt}\selectfont, inner xsep=0.080cm, inner ysep=0.085cm, rounded corners=0.020cm] at (0.25,0.5) {1};
\node [font=\fontsize{18.2pt}{23.7pt}\selectfont, inner xsep=0.080cm, inner ysep=0.085cm, rounded corners=0.020cm] at (3.25,0.5) {1};
\node [font=\fontsize{18.2pt}{23.7pt}\selectfont, inner xsep=0.080cm, inner ysep=0.085cm, rounded corners=0.020cm] at (2.25,0.5) {1};
\node [font=\fontsize{18.2pt}{23.7pt}\selectfont, inner xsep=0.080cm, inner ysep=0.085cm, rounded corners=0.020cm] at (6.25,0.5) {1};
\end{circuitikz}
}%
\caption{Nim position that can be written as $G_1+H_1+H_2$ with $g(G_1)=g(H_1)=g(H_2)=1$ by taking $G_1\equiv*7+*6$, $H_1\equiv*5+*4$ and $H_2\equiv*$}
\label{fig:2}
\end{figure}

\begin{figure}[!ht]
\centering
\resizebox{1\textwidth}{!}{%
\begin{circuitikz}
\tikzstyle{every node}=[font=\fontsize{18.2pt}{23.7pt}\selectfont]
\draw  (0,0) rectangle (1,1);
\draw  (1,0) rectangle (2,1);
\draw  (2,0) rectangle (3,1);
\draw  (3,0) rectangle (4,1);
\draw  (4,0) rectangle (5,1);
\draw  (5,0) rectangle (6,1);
\draw  (6,0) rectangle (7,1);
\draw [ fill={rgb,255:red,0; OliveGreen,0; blue,0}, fill opacity=1] (7,0) rectangle (8,1);
\node [font=\fontsize{18.2pt}{23.7pt}\selectfont, inner xsep=0.080cm, inner ysep=0.085cm, rounded corners=0.020cm] at (0.25,0.5) {1};
\node [font=\fontsize{18.2pt}{23.7pt}\selectfont, inner xsep=0.080cm, inner ysep=0.085cm, rounded corners=0.020cm] at (6.25,0.5) {2};
\end{circuitikz}
}%
\caption{Nim position that can be written as $G_1+H_1+H_2$ with $g(H_1)=g(H_2)=1$ and $g(G_1)\ne1$ by taking $G_1\equiv*7$, $H_1\equiv*$ and $H_2\equiv*$}
\label{fig:3}
\end{figure}

For example we know that the game represented in Figure~\ref{fig:2} with a pass is a $\mathcal{P}$-position and the one represented in Figure~\ref{fig:3} with a pass is an $\mathcal{N}$-position.

\begin{proof}
We prove the claim by induction on $G$.

From $G\equiv G_1+H_1+H_2$, first player can play to:
\begin{itemize}
    \item $G_1'+H_1+H_2$ with $G_1'$ an option from $G_1$,
    \item $G_1+H_1'+H_2$ with $H_1'$ an option from $H_1$,
    \item $G_1+H_1+H_2'$ with $H_2'$ an option from $H_2$.
\end{itemize}

By symmetry for the rest of the proof we will ignore the move on $H_2$ to $H_2'$ as $H_1$ and $H_2$ play the same role.

If $g(G)=1$, then $g(G_1'+H_1+H_2)\ne 1$ and by induction $s(G_1'+H_1+H_2)\ne 1$. 
If $g(H_1')>1$ there is a move from $G_1+H_1'+H_2$ to $G_1+H_1''+H_2$ with $g(H_1'')=1$ and $H_1''$ is an option of $H_1'$, but by induction $s(G_1+H_1''+H_2)=1$ and then $s(G_1+H_1'+H_2)\neq1$. If $g(H_1')=0$, then $g(G_1+H_1'+H_2)=0$ and $s(G_1+H_1'+H_2)=0$. In addition, from the mex rule, $G_1+H_1+H_2$ has at least one option whose Grundy value is $0$. Thus $G_1+H_1+H_2$ has an option whose split value is $0$. By the mex rule and $0$ cannot be an option of $G$ as $G$ has at least two non empty components, we have that $s(G)=1$.

If  $g(G)\ne1$, then either $g(G)=0$ and then $s(G)=0$ or $g(G_1)>1$ and there exists an $G_1'$ an option of $G_1$ as $g(G'_1)=1$ but then by induction $s(G'_1+H_1+H_2)=1$ and then $s(G)\ne1$. 
    
\end{proof}

\section{Results for split sum of Nim game}
\label{sec_split_and_nim}

In this section, we show some interesting results on Nim game and split sum, which allows us to compute easily the $\mathcal{P}$-positions of some Nim games.

For this part we write $n*m$ for $*m+*m+\cdot \cdot \cdot+*m$ with $n$ times $*m$.

\begin{theorem}
For any positive integers $n$ and $m$,
$s(*n+*m)=(n-1)\oplus (m-1)$.
\end{theorem}

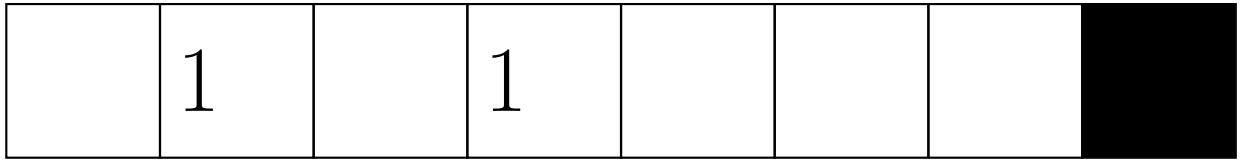
\begin{figure}[!ht]
\centering
\resizebox{1\textwidth}{!}{%
\begin{circuitikz}
\tikzstyle{every node}=[font=\fontsize{18.2pt}{23.7pt}\selectfont]
\draw  (0,0) rectangle (1,1);
\draw  (1,0) rectangle (2,1);
\draw  (2,0) rectangle (3,1);
\draw  (3,0) rectangle (4,1);
\draw  (4,0) rectangle (5,1);
\draw  (5,0) rectangle (6,1);
\draw  (6,0) rectangle (7,1);
\draw [ fill={rgb,255:red,0; green,0; blue,0}, fill opacity=1] (7,0) rectangle (8,1);
\node [font=\fontsize{18.2pt}{23.7pt}\selectfont, inner xsep=0.080cm, inner ysep=0.085cm, rounded corners=0.020cm] at (1.25,0.5) {1};
\node [font=\fontsize{18.2pt}{23.7pt}\selectfont, inner xsep=0.080cm, inner ysep=0.085cm, rounded corners=0.020cm] at (3.25,0.5) {1};
\end{circuitikz}
}%
\caption{Example of a Nim game with two piles}
\label{fig:exemple_two_pile}
\end{figure}

For example in the game represented in Figure~\ref{fig:exemple_two_pile}, the split value is equal to $(4-1)\oplus(6-1)=3\oplus5=6$.

\begin{proof}
By Theorem~\ref{mexrule}, $s(*n)=\infty$ and $s(*1+*1)=0$.

By induction on $n+m$, as $s(*k)=\infty$ for all $k$, we have $s(*n+*m)={\rm mex}(\{s(*i+*m),s(*n+*j)|0 \le i < n,0 \le j < m\})={\rm mex}(\{(i-1)\oplus (m-1), (n-1)\oplus (j-1)|0 < i < n,0 < j < m\})=g(*(n-1)+*(m-1))=(n-1)\oplus (m-1)$.
 \end{proof}

\begin{theorem}
\label{G+*} For all multiple-pile Nim-games $G$ and for all $n \geq 2b(G)+1$,
\[
(G+n*)\circ G = 0 \textrm{ or } (G+(n+1)*)\circ G = 0.
\]
\end{theorem}

\begin{figure}[!ht]
\centering
\resizebox{1\textwidth}{!}{%
\begin{circuitikz}
\tikzstyle{every node}=[font=\fontsize{18.2pt}{23.7pt}\selectfont]
\draw  (0,0) rectangle (1,1);
\draw  (1,0) rectangle (2,1);
\draw  (2,0) rectangle (3,1);
\draw  (3,0) rectangle (4,1);
\draw  (4,0) rectangle (5,1);
\draw  (5,0) rectangle (6,1);
\draw  (6,0) rectangle (7,1);
\draw [ fill={rgb,255:red,0; green,0; blue,0}, fill opacity=1] (7,0) rectangle (8,1);
\node [font=\fontsize{18.2pt}{23.7pt}\selectfont, inner xsep=0.080cm, inner ysep=0.085cm, rounded corners=0.020cm] at (0.25,0.5) {1};
\node [font=\fontsize{18.2pt}{23.7pt}\selectfont, inner xsep=0.080cm, inner ysep=0.085cm, rounded corners=0.020cm] at (2.25,0.5) {2};
\node [font=\fontsize{18.2pt}{23.7pt}\selectfont, inner xsep=0.080cm, inner ysep=0.085cm, rounded corners=0.020cm] at (6.25,0.5) {17};
\node [font=\fontsize{18.2pt}{23.7pt}\selectfont, color=OliveGreen, inner xsep=0.080cm, inner ysep=0.085cm, rounded corners=0.020cm] at (0.75,0.5) {1};
\node [font=\fontsize{18.2pt}{23.7pt}\selectfont, color=OliveGreen, inner xsep=0.080cm, inner ysep=0.085cm, rounded corners=0.020cm] at (2.75,0.5) {2};
\end{circuitikz}
}%

\resizebox{1\textwidth}{!}{%
\begin{circuitikz}
\tikzstyle{every node}=[font=\fontsize{18.2pt}{23.7pt}\selectfont]
\draw  (0,0) rectangle (1,1);
\draw  (1,0) rectangle (2,1);
\draw  (2,0) rectangle (3,1);
\draw  (3,0) rectangle (4,1);
\draw  (4,0) rectangle (5,1);
\draw  (5,0) rectangle (6,1);
\draw  (6,0) rectangle (7,1);
\draw [ fill={rgb,255:red,0; green,0; blue,0}, fill opacity=1] (7,0) rectangle (8,1);
\node [font=\fontsize{18.2pt}{23.7pt}\selectfont, inner xsep=0.080cm, inner ysep=0.085cm, rounded corners=0.020cm] at (0.25,0.5) {1};
\node [font=\fontsize{18.2pt}{23.7pt}\selectfont, inner xsep=0.080cm, inner ysep=0.085cm, rounded corners=0.020cm] at (2.25,0.5) {2};
\node [font=\fontsize{18.2pt}{23.7pt}\selectfont, inner xsep=0.080cm, inner ysep=0.085cm, rounded corners=0.020cm] at (6.25,0.5) {18};
\node [font=\fontsize{18.2pt}{23.7pt}\selectfont, color=OliveGreen, inner xsep=0.080cm, inner ysep=0.085cm, rounded corners=0.020cm] at (0.75,0.5) {1};
\node [font=\fontsize{18.2pt}{23.7pt}\selectfont, color=OliveGreen, inner xsep=0.080cm, inner ysep=0.085cm, rounded corners=0.020cm] at (2.75,0.5) {2};
\end{circuitikz}
}%
\caption{Example of two games where Theorem~\ref{G+*} gives us that one of them is a $\mathcal{P}$-position}
\label{fig:exemple_many_star}
\end{figure}

For example this theorem gives us that one of the game represented in Figure~\ref{fig:exemple_many_star} is a $\mathcal{P}$-position.

This theorem can help us find $\mathcal{P}$-positions of Nim game as if we add some $*$ to the game we know that by some time we will have the split value equal to the Grundy value of the Nim game. 

Even if this theorem only gives us possibility of being a $\mathcal{P}$-position, it still gives important information for a game able to go to both position would be an  $\mathcal{N}$-position; for example $(G+n*+*2)\circ G$ is an $\mathcal{N}$-position if $n\ge 2b(G)+1$ and if we know that one of them is not a $\mathcal{P}$-position it gives immediately the result for the other. 
Also, when we reach a $\mathcal{P}$-position where this theorem can be applied, we know that as long as the parity of the number of stars is kept, the position is still a $\mathcal{P}$-position even if we add some stars.


We use the following Lemmas~\ref{lemma:game_at_one} and~\ref{GtoG+*} to prove Theorem~\ref{G+*}.

\begin{lemma}
\label{lemma:game_at_one}
If G is a Nim game and $G'\in G$ then either 
\begin{equation}
\label{hypotesis:nim1bis}
    g(G') = g(G) \oplus 1
\end{equation}
or there exists $G_1'\in G$ such that $g(G') = g(G_1') \oplus 1$ and there is a move either from $G'$ to $G_1'$ or from $ G_1'$ to $G'$.
\end{lemma}

\begin{proof}
Suppose that Equation~\ref{hypotesis:nim1bis} does not hold.

Because $G$ is a Nim game, we know that moving from $G$ to $G'$ is by removing a certain number of tokens from a Nim head. 

If we have left in this Nim head an odd number of tokens then by removing one more token we have a game $G'_1$ accessible from $G$ and $G'$ and because the only change is on the last bit we know that $g(G') = g(G_1') \oplus 1$.

If we have left in this Nim head an even number of tokens then by removing one less we have a game $G'_1$ accessible from $G$ because it cannot be $G$ by Equation~\ref{hypotesis:nim1bis}, from this game we can access $G'$ and we have $g(G') = g(G_1') \oplus 1$ then the theorem holds.
 \end{proof}

\begin{lemma}
\label{GtoG+*}
Let $b$ be a non negative integer. If for all Nim game $G$ such that $b(G) \leq b$, the statement of Theorem~\ref{G+*} holds then for all $H+*$ with a Nim game $H$ whose birthday equals to $b$, for all $n \geq 2b+1$, $(H+*+n*)\circ (H+*)=0$ or $(H+*+(n+1)*)\circ (H+*)=0$ holds.
\end{lemma}

\begin{proof}
We use induction on $b$.

From the assumption, for any $n \ge 2b+1$, one of $(H + (n+1)*) \circ H = 0$ or $(H + n*) \circ H = 0$ holds.

Suppose that $(H+(n+1)*)\circ H= 0$.

If $(H+*+(n+1)*)\circ (H+*) \ne 0$ then one of the following equations holds:

\begin{equation}
\label{eq:1}
(H+*+n*)\circ (H+*) = 0  ,
\end{equation}
\begin{equation}
\label{eq:2}
(H'+*+(n+1)*)\circ (H+*) = 0  \text{ with $H'$  an option of $H$,}
\end{equation}
\begin{equation}
\label{eq:3}
        (H+*+(n+1)*)\circ (H'+*) = 0 \text{ with $H'$  an option of $H$,}
\end{equation}
\begin{equation}
\label{eq:4}
        (H+*+(n+1)*)\circ H= 0  .
\end{equation}

Equation~\ref{eq:1} gives us a contradiction as from $(H+*+n*)\circ (H+*)$ we can go to $(H+*+n*)\circ H=0$.

Equation~\ref{eq:4} gives us a contradiction by answering in $*$ and then is not equal to $0$.

We separate the two following cases if this equation holds or not:
\begin{equation}
\label{hypotesis:nim1}
    g(H') = g(H) \oplus 1
\end{equation}

If Equation~\ref{hypotesis:nim1} does not hold then by Lemma~\ref{lemma:game_at_one} there exists an $H_1' \in H$ and $g(H') = g(H_1') \oplus 1$ and there is a move either from $H'$ to $H_1'$ or from $ H_1'$ to $H'$.\\

We have three cases:
\begin{itemize}
    \item [(i)] Equation~\ref{eq:2} holds and Equation~\ref{hypotesis:nim1} does not hold: We both have $(H'+*+(n+1)*)\circ (H'+*) \ne 0$ and $(H'+*+(n+1)*)\circ H' \ne 0$. 
By hypothesis, first inequality gives us $(H'+*+(n+2)*)\circ (H'+*) = 0$ and second inequality implies $(H'+*+(n+2)*)\circ H'= 0$ if $n+1 \ge 2b + 1+1 >  2b(H'+*) + 1$.
However, there is a move from $(H'+*+(n+2)*)\circ (H'+*)$ to $(H'+*+(n+2)*)\circ H'$ then one of them is not $0$ and we have a contradiction. 
    \item [(ii)] Equation~\ref{eq:3} holds and Equation~\ref{hypotesis:nim1} does not hold: We both have $(H'+*+(n+1)*)\circ (H'+*) \ne 0$ and $(H_1'+*+(n+1)*)\circ (H'+*) \ne 0$. 
By hypothesis, first inequality gives us $(H'+*+(n+2)*)\circ (H'+*) = 0$ and second inequality implies $(H_1'+(n+3)*)\circ (H'+*)= 0$ (because $H'+*=H_1'$) if $n \ge 2b+1$. But this is impossible because there is a move from $H'$ to $H_1'$ or from $ H_1'$ to $H'$. Then we have a contradiction. 
    \item [(iii)] Equation~\ref{eq:2} or Equation~\ref{eq:3} holds and Equation~\ref{hypotesis:nim1} holds: Either $(H'+*+(n+1)*)\circ (H'+*) = 0$ and then it contradicts Equation~\ref{eq:2} or Equation~\ref{eq:3}, or $(H'+*+n*)\circ (H'+*) = 0$.
But we have supposed that $(H+(n+1)*)\circ H= 0$ and $(H'+*+n*)\circ (H'+*) = (H'+*+n*)\circ H \ne 0$. We have a contradiction. 
\end{itemize}

If $(H+n*)\circ H= 0$, except for the case (iii), the proof is verbatim  from the proof for $(H+(n+1)*)\circ H= 0$ after subtracting a $*$ in the left part of every corresponding equations.
If $(H'+*+n*)\circ (H+*) = 0$ or $(H+*+n*)\circ (H'+*) = 0$ with $H'$ an option of $H$ such as $g(H') = g(H) \oplus 1$, then $(H'+*+(n+1)*)\circ (H'+*) = 0$ as $(H'+*+n*)\circ (H'+*) = 0$ contradicts $(H'+*+n*)\circ (H+*) = 0$ or $(H+*+n*)\circ (H'+*) = 0$ and $n\ge 2b(H)+1\ge2b(H'+*)+1$.
However, we have supposed that $(H+n*)\circ H= 0$ which means that $(H+(n+2)*)\circ H= 0$ and $(H'+*+(n+1)*)\circ (H'+*) = (H'+(n+2)*)\circ H \ne 0$. We have a contradiction. 
 \end{proof}

\begin{proof}[Proof of Theorem~\ref{G+*}] 

We proceed by induction on $b(G)$.

If $(G+(n+1)*)\circ G \ne 0$ then one of the following equation holds:

\begin{equation}
\label{eq:5}
(G+n*)\circ G = 0,  
\end{equation}
\begin{equation}
\label{eq:6}
(G'+(n+1)*)\circ G = 0  \text{ with $G'$ an option of $G$,}
\end{equation}
\begin{equation}
\label{eq:7}
        (G+(n+1)*)\circ G' = 0 \text{ with $G'$ an option of $G$.}
\end{equation}

Equation~\ref{eq:5} gives us what we want.

We separate the two following cases if this equation holds or not:
\begin{equation}
\label{hypotesis:nim2}
    g(G') = g(G) \oplus 1
\end{equation}

If Equation~\ref{hypotesis:nim2} does not hold then by Lemma~\ref{lemma:game_at_one} there exists an $G_1' \in G$ and $g(G') = g(G_1') \oplus 1$ such that is a move either from $G'$ to $G_1'$ or from $ G_1'$ to $G'$.\\

Then, if Equation~\ref{eq:6} holds and Equation~\ref{hypotesis:nim2} does not hold, we both have $(G'+(n+1)*)\circ G' \ne 0$ and $(G'+(n+1)*)\circ (G'+*) \ne 0$. 
By induction hypothesis, first inequality gives us $(G'+n*)\circ G' = 0$ and second inequality implies by Lemma~\ref{GtoG+*} $(G'+n*)\circ (G'+*)= 0$.
But there is a move from $(G'+n*)\circ (G'+*)$ to $(G'+n*)\circ G'$ then one of them is not $0$ and we have a contradiction. \\

If Equation~\ref{eq:7} holds and Equation~\ref{hypotesis:nim2} does not hold, we both have $(G'+(n+1)*)\circ G' \ne 0$ and $(G_1'+(n+1)*)\circ G' \ne 0$. 
By induction hypothesis, first inequality gives us $(G'+n*)\circ G' = 0$ and second inequality implies by Lemma~\ref{GtoG+*} $(G_1'+n*)\circ G'= 0$ (because $G'=G_1'+*$). But this is impossible because there is a move from $G'$ to $G_1'$ or from $ G_1'$ to $G'$. Then we have a contradiction. \\

To finish if Equation~\ref{eq:6} or Equation~\ref{eq:7} holds and Equation~\ref{hypotesis:nim2} holds, if $(G+n*)\circ G \ne 0$ then either $(G+(n-1)*)\circ G = 0$ or $(G_2'+n*)\circ G = 0$ or $(G+n*)\circ G_2' = 0$ with $G_2'$ an option of $G$ but by a similar discussion to those in the two previous paragraphs by subtracting a $*$ in the left part of every corresponding equations, we have a contradiction or we have one of the following equations:
\begin{itemize}
    \item $(G+(n-1)*)\circ G = 0$,
    \item there is a $G_2'$ option of $G$ such that $g(G_2') = g(G) \oplus 1$ and $(G_2'+n*)\circ G_2' \ne 0$. 
\end{itemize}

However, in the latter case, $G'$ and $G_2' $ do not play on the same Nim head because both of them have the same Grundy value and are different option from $G$ as $(G_2'+n*)\circ G_2' \ne 0$ and $(G'+n*)\circ G'=0$ by induction on $b(G)$ and $(G'+(n+1)*)\circ G'\ne 0$. As $g(G')=g(G)\oplus 1$ we know that the move from $G$ to $G'$ has only change the rightmost bit in the binary writing of the size of one pile because if it is not the case the difference between $g(G)$ and $g(G')$ would be greater. The same argument works for $G_2'$. We can then from $G'$ do the same move as from $G$ to $G_2'$ which gives us an option $G''$ of $G'$ and $G_2'$ that have the same Grundy value as $G$. However, as $(G'+n*)\circ G'=0$ and either Equation~\ref{eq:6} or Equation~\ref{eq:7} holds, we have $(G''+n*)\circ G'=(G''+*+(n-1)*)\circ (G''+*) \ne 0$ and as $(G_2'+(n+1)*)\circ G_2'=0$ by induction on $b(G)$, we have $(G''+(n+1)*)\circ G_2'=(G''+*+n*)\circ (G''+*) \ne 0$ because $G''+*=G'=G_2'$ then this contradicts that $G''+*$ verify the theorem. \\

It follows that either $(G+(n+1)*)\circ G = 0$ or $(G+n*)\circ G = 0$ or $(G+(n-1)*)\circ G = 0$. However the two first cases give us the theorem. If we have $(G+(n-1)*)\circ G = 0$ and not $(G+n*)\circ G = 0$ or $(G+(n+1)*)\circ G = 0$ then by the previous argument there is a $G'$ such that $g(G') = g(G) \oplus 1$, and $(G'+n*)\circ G' = 0$. By playing from $G$ to $G'$ we arrive to $(G'+(n-1)*)\circ G=(G'+*+(n-2)*)\circ (G'+*)\neq 0$ and we have $(G'+*+(n-1)*)\circ (G'+*)=0$ as $b(G')\le b(G)-1$ so $n-2\ge 2b(G)+1-2 \ge 2b(G')+1$ and we can use Lemma~\ref{GtoG+*}. However, $(G'+*+(n-1)*)\circ (G'+*)$ has an option $(G'+n*)\circ G'$ then both cannot be equal to $0$ and we have a contradiction. Then either $(G+(n+1)*)\circ G = 0$ or $(G+n*)\circ G = 0$. 
\end{proof}

\section{Modulo Strong Equality}
\label{sec_modulo_strong_equality}

The usual definition of equality in combinatorial game theory does not work well with split sum as two games can be equal but can have different split sum. Then we seek a stronger property of equality. 

With that in mind we try to define a strong equality that will imply the same behavior with split sum.

We observe that this strong equality is too strong to find anything. Then we introduce this equality but modulo the nimbers.

\begin{definition}
Let $G$ and $H$ be impartial games, $G$ is strongly equal to $H$ if and only if for all impartial games $X$ and $Y$, the outcomes of $(G+X)\circ Y$ and $(H+X)\circ Y$ are the same.

$G$ and $H$ are strongly equal modulo a set $A$ of games if and only if for all games $X$ in $A$ and $Y$ in $A$, the outcomes of $(G+X)\circ Y$ and $(H+X)\circ Y$ are the same.
\end{definition}

We analyze the games by the strong equality modulo the nimbers.

This strong equality modulo the nimbers is exactly defined by a sequence that starts with the Grundy value and acts as a generalization of Grundy value and split value.

\begin{definition}
For an impartial game $G$, $g_G(n)$ is the integer such that $(G+*g_G(n))\circ *n=0$. We name $g_G$ the extended Grundy function of $G$.

$s_G(n)$ is the integer, if it exists and is unique, such that $(G+*n)\circ *s_G(n)=0$. If it does not exist or is not unique, $s_G(n)=\infty$. We name $s_G$ the extended split function of $G$.

\end{definition}

For all $G$ and $n$ there exists an unique $g_G(n)$.

If it exists, it is unique because for all $m_1$ and $m_2$ either from $(G+*m_1)\circ *n$ we can go to $(G+*m_2)\circ *n$ or from $(G+*m_2)\circ *n$ we can go to $(G+*m_1)\circ *n$ then only one of them can be a $\mathcal{P}$-position.

We can confirm that it exists by induction, as for all $m$, from $(G+*m)\circ *n$ all options are like the following $(G+*m_1)\circ *n$ with $m>m_1$, $(G+*m)\circ *n_1$ with $n>n_1$ or $(G'+*m)\circ *n$ with $G'$ an option of $G$. However if a position of the type $(G+*m_1)\circ *n=0$ with $m>m_1$ is a $\mathcal{P}$-position then we have $g_G(n)=m_1$. For each $n_1<n$ and $G'$ option of $G$, there is only one $m$ such as $(G+*m)\circ *n_1=0$ and $(G'+*m)\circ *n=0$ but as there only a finite number of possible $n_1$ and $G'$ then there should have an $m$ such as $(G+*m)\circ *n$ cannot move to a $\mathcal{P}$-position and then is a $\mathcal{P}$-position.

\begin{lemma}
\label{lemma:infty}
    For any impartial short game $G$ and non negative integer $n$, if $s_G(n)=\infty$ then we have $G\equiv 0$ or we have $0\in G$ and $n=0$.

    $(G+*n)\circ *x=0$ has multiple solution if and only if $G\equiv 0$ and $n=0$.
\end{lemma}

\begin{proof}
    We take a short game $G$ and a non negative integer $n$. Suppose that $s_G(n)=\infty$. 

    If there are multiple integers $x$ such that $(G+*n)\circ *x=0$, we take $x_1$ and $x_2$ two of the $x$ such as $(G+*n)\circ *x=0$. As both $(G+*n)\circ *x_1$ and $(G+*n)\circ *x_2$ are $\mathcal{P}$-positions then there is no move from one to the other which is only possible if $G+*n$ is the empty game. Thus, $G$ is the empty game and $n=0$.

    If there is no non negative integer $x$ such as $(G+*n)\circ *x=0$, then for all $x$ we can find an $G'_x$ an option of $G+*n$ such as $G'_x\circ *x=0$ but as $G+*n$ has only a finite number of options, it means that for some $x_1,x_2$, $G'_{x_1}\equiv G'_{x_2}$. However, this means that $s_{G'_{x_1}}(0)=\infty$ and by the previous paragraph of the proof, $G'_{x_1}\equiv 0$. Thus, ($0\in G$ and $n=0$) or $G\equiv 0$.
    
\end{proof}

\begin{theorem}
    For two impartial short games $G$ and $H$ the following are equivalent:
    \begin{itemize}
        \item for all $n$, $g_G(n)=g_H(n)$,
        \item for all $m$, $s_G(m)=s_H(m)$,
        \item $G$ is strongly equal to $H$ modulo the nimbers.
    \end{itemize}
\end{theorem}

\begin{proof}
    Suppose that for all $n$, $g_G(n)=g_H(n)$.

We take an integer $m$.

If $s_G(m)\ne \infty$ and $s_H(m)\ne \infty$, we have $(G+*m)\circ *s_G(m)=0$. 
However, then $g_G(s_G(m))=m$ which implies that $g_H(s_G(m))=m$ and $(H+*m)\circ *s_G(m)=0$. Then we have $s_H(m)=s_G(m)$. 

We treat now the case where $s_G(m)= \infty$ or $s_H(m)= \infty$, without loss of generality we assume that $s_G(m)= \infty$.
 We have either there is no integer $x$ such as $(G+*m)\circ *x=0$ or there are multiple integers $x$ such that $(G+*m)\circ *x=0$.

If there is no non negative integer $x$ such as $(G+*m)\circ *x=0$ then for all $x$, $g_G(x)\ne m$ then $g_H(x) \ne m$ and $(H+*m)\circ x \ne 0$. Then $s_G(m)=s_H(m)=\infty$.

If there are multiple non negative integers $x$ such that $(G+*m)\circ *x=0$ then for each of these non negative integers $g_G(x)=g_H(x)= m$ and $(H+*m)\circ *x = 0$. Then $s_G(m)=s_H(m)=\infty$.

Thus, if for all $n$, $g_G(n)=g_H(n)$ then for all $m$, $s_G(m)=s_H(m)$. \\

Suppose that for all $m$, $s_G(m)=s_H(m)$.

We take two non negative integers $n$ and $m_1$.

If $s_G(m_1)=n$, then $(G+*m_1)\circ*n$ is a $\mathcal{P}$-position and because $s_H(m_1)=n$, $(H+*m_1)\circ*n$ is a $\mathcal{P}$-position.

If $s_G(m_1)\ne \infty$ and $s_G(m_1)\neq n$, then $(G+*m_1)\circ*n$ is an $\mathcal{N}$-position because we can go either from  $(G+*m_1)\circ*n$ to $(G+*m_1)\circ*s(m_1)$ or from $(G+*m_1)\circ*s(m_1)$ to $(G+*m_1)\circ*n$. By similar argument $(H+*m_1)\circ*n$ is an $\mathcal{N}$-position.

If $s_G(m_1)=s_H(m_1)=\infty$, we have two cases:

We now treat the case where there are multiple $x$ such as $(G+*m_1)\circ *x=0$ or multiple $x$ such as $(H+*m_1)\circ *x=0$, without loss of generality we consider only the first case. Thus by Lemma~\ref{lemma:infty}, $G\equiv 0$ and for all $m$, $s_G(m)=\infty$ then $s_H(m)=\infty$ and by Lemma~\ref{lemma:infty}, $H\equiv 0$. Thus, $H\equiv G \equiv 0$ and $(G+m_1)\circ n$ and $(H+m_1)\circ n$ have both the same outcome.



If there is no non negative integer $x$ such as $(G+*m_1)\circ *x=0$ or $(H+*m_1)\circ *x=0$, then such as $(G+*m_1)\circ *n$ and $(H+*m_1)\circ *n$ are $\mathcal{N}$-position. 

Thus, if for all $m$, $s_G(m)=s_H(m)$, then $G$ is strongly equals to $H$ modulo nimbers. \\

Suppose that $G$ is strongly equals to $H$ modulo nimbers.

Then for all $n$, as $(G+*g_G(n))\circ *n$ is a $\mathcal{P}$-position then $(H+*g_G(n))\circ *n$ is also a $\mathcal{P}$-position and thus, $g_G(n)=g_H(n)$.
\end{proof}

\begin{theorem}
\label{mexrule2}
For all non-empty $G$, and for all $n,i$,
\[g_G(n) =
    {\rm mex} (\text{\{}g_H(j) \text{, with } (H\in G \text{ and } j=n) \text{ or } (H\equiv G \text{ and } j < n) \text{\}})
\]

\[s_G(i) =
 \begin{cases} 
      \infty & \text{ if } 0 \in G \text{ and } i=0 \\
       {\rm mex} (\text{\{}s_H(j) \text{, with } (H\in G \text{ and } j=i) \text{ or } (H \equiv G \text{ and } j < i) \text{\}})  &  \text{ otherwise} 
   \end{cases}
\]

\end{theorem}

\begin{proof}
We take $G$ a non-empty game.

We take $m$ an integer.
If there is a $\mathcal{P}$-position among $(G'+*m)\circ *n$ or $(G+*m)\circ *i$ with $G'$ an option of $G$ and $i < n$ then we know that $(G+*m)\circ *n \ne 0$ then $g_G(n) \ne g_{G'}(n)$ and $g_G(n) \ne g_{G}(i)$. 

On the other hand, if all these positions are $\mathcal{N}$-positions then we know that either $(G+*m)\circ *n = 0$ or there exists $j<m$, $(G+*j)\circ *n = 0$ because if not all positions that can be reached from this game are $\mathcal{N}$-positions.

Thus, $g_G(n) = {\rm mex} (\text{\{}g_H(j) \text{, with } (H\in G \text{ and } j=n) \text{ or } (H\equiv G \text{ and } j < n )\text{\}})$. \\

If $ 0 \in G \text{ and } i=0 $, then from $G\circ *n$ there is always a winning move to $0$. Thus, $s_G(0)=\infty$. 

If $0\notin G$ or $i\ne 0$, then by Lemma~\ref{lemma:infty}, $s_G(i)\ne \infty$.  We take $n$ an integer.
If there is a $\mathcal{P}$-position among $(G'+*i)\circ *n$ or $(G+*m)\circ *n$ with $G'$ an option of $G$ and $m < i$ then we know that $(G+*i)\circ *n \ne 0$ then as $s_G(i)\ne \infty$, $s_G(i) \ne s_{G'}(i)$ and $s_G(m) \ne s_{G}(i)$. 

On the other hand, if all these positions are $\mathcal{N}$-positions then we know that either $(G+*i)\circ *n = 0$ or there exists $j<n$, $(G+*i)\circ *j = 0$ because if not all positions that can be reached from this game are $\mathcal{N}$-positions. 

Thus, $s_G(i) = {\rm mex} (\text{\{}s_H(j) \text{, with } (H\in G \text{ and } j=i) \text{ or } (H\equiv G \text{ and } j < i) \text{\}})$.  \end{proof}

These sequences follow some rules and are in fact inverse of each other.

We name the set of all non negative integers $\mathbb N$ and the set of all positives integers $\mathbb N^*$.

\begin{proposition}
\label{proposition:bijection}
If $G$ is not empty and does not have the empty game as an option, $g_G$ is a bijection from $\mathbb N$ to $\mathbb N$ with inverse $s_G$.

If $G$ is not empty and  have the empty game as an option, $g_G$ is a bijection from $\mathbb N$ to $\mathbb N^*$ with inverse $s_G$.
\end{proposition}

\begin{proof}
We take $G$ a non-empty game.

By the mex rule, we have the injectivity of $g_G$.

If $G$ does not have the empty game as an option:

Suppose that $g_G$ is not surjective. We take the minimum $m$ integer not reached by $g_G$.

At a finite $n$, we have for all $i<m$ there exists a $j<n$ such that $g_G(j)=i$. Then for all $k>n$, according to the mex rule, if there does not exist a $G'\in G$ such that $g_{G'}(k)=m$ we have $g_G(k)=m$ but there are only a finite number of options from $G$ and $g_{G'}$ is injective then it is possible to reach $m$ and we have a contradiction.
Then $g_G$ is surjective.

If $G$ has the empty game as an option:

$g_G(n)\ne 0$ for all $n$ because $g_0(n)=0$ for all $n$. Then $0$ is non-reachable by $g_G$.

Suppose that $g_G$ is not surjective. We take the minimum $m \in \mathbb N$ integer not reached by $g_G$.

At a finite $n$, we have for all $i<m$ there exists a $j<n$ such that $g_G(j)=i$. Then for all $k>n$, according to the mex rule, if there does not exist a $G'\in G$ that is not the empty game and $g_{G'}(k)=m$ we have $g_G(k)=m$ but there are only a finite number of options from $G$ and $g_{G'}$ is injective then it is possible to reach $m$ and we have a contradiction.
Then $g_G$ is surjective.

As $(G+*g_G(n)) \circ *n=0$ we have $s_G(g_G(n))=n$ then $s_G$ is the inverse of $g_G$.
 \end{proof}

We can in fact find for all sequence that are bijective from $\mathbb N$ to $\mathbb N$ or $\mathbb N$ to $\mathbb N^*$, a game where the first $n$ terms of the sequence and its extended Grundy function are the same and a game where the first $n$ terms of the sequence and its extended split function are the same.

\begin{theorem}
For all $n$, $a_0$, $\ldots$, $a_n$ such that for all $i, j\leq n$, if $i \ne j$, $a_i \ne a_j$, there exists $G$ such that for all $i \leq n$, $g_G(i)=a_i$.

For all $n$, $a_0$, $\ldots$, $a_n$ such that for all $i,j\leq n$, if $i \ne j$ $a_i \ne a_j$, there exists $G$ such that for all $i \leq n$, $s_G(i)=a_i$.
\end{theorem}

We know that there exists a bijection from $\mathbb N$ to $\mathbb N$ that corresponds to any $g_G$ because there is an uncountable number of bijections but only a countable number of short games. 

\begin{proof}
We first prove that the theorem works for the extended Grundy function.

We proceed by induction on $n$ then on $a_n$.

First the theorem is true for $n=0$ because $g_G(0)$ is the Grundy value of $G$.

We separate the cases where $a_1=0$ and $a_1=1$ with $n=1$, as these cases need a variation in the argument.

For $n=1$, if $a_1=0$ we can take the game $H=\{F_k|0\le k<a_0\}$ with $F_1=*$ and  $F_k=*k+*+*$ for all $k<a_0$ and $k\ne 1$, by Theorem~\ref{theorem:s=1} for all these games $g_{F_k}(1)\ne 0$ and $g_{F_k}(0)=g(F_k)=k$. Thus, $g_H(0)=a_0$ and $g_H(1)=0$  by the mex rule.

If $a_1=1$ and $a_0\ne 0$ we can take the game $H=\{0,F_k|0< k<a_0\}$ with $F_k$ for all $0<k<a_0$ such that $g_{F_k}(0)=k$ and $g_{F_k}(1)=0$. Thus, $g_H(0)=a_0$ and $g_H(1)=1$  by the mex rule.

If $a_1=1$ and $a_0= 0$ we can take the game $H=\{F\}$ with $F$ such that $g_{F}(0)>0$ and $g_{F}(1)=0$. Thus, $g_H(0)=0$ and $g_H(1)=1$  by the mex rule.

Suppose the theorem is true for $n-1$:

If $a_n=0$, we want a game $H$ such as $g_H(n')=a_{n'}$ for all $n'\le n$. To achieve that, we take $H=\{F_{i,k,\ell}|0\le i, k <n, i\ne k,0< \ell < a_k \}$ with $F_{i,k,\ell}$ such as $g_{F_{i,k,\ell}}(i)=0$, $g_{F_{i,k,\ell}}(k)=\ell$ and $g_{F_{i,k,\ell}}(j)>a_j$ for all $j<n$ with $j\ne k,i$. We can create all these games by induction on $n$. We then have a game $H$ such as $g_H(n')=a_{n'}$ for all $n'<n$ and $g_H(n)=0$ by applying the mex rule as $g_{F_{i,k,\ell}}(n)\ne g_{F_{i,k,\ell}}(i)=0$.


If $a_n=1$, we want a game $H$ such as $g_H(n')=a_{n'}$ for all $n'\le n$. We take the game  $H=\{G,F_{i,k,\ell}|0\le i,k <n, i\ne k, 0\le \ell<a_k, \ell \ne 1\}$ with $F_{i,k,\ell}$ such that $g_{F_{i,k,\ell}}(i)=1$, $g_{F_{i,k,\ell}}(k)=\ell$ and $g_{F_{i,k,\ell}}(j)>a_j$ for all $j<n$ with $j\ne k,i$ and $G$ such as $g_G(j)>a_j$ for all $j<n$ and $g_G(n)=0$. Such $G$ and $F_{i,k,\ell}$ exist by induction hypothesis. For this game we have $g_H(n')=a_{n'}$ for all $n'< n$ and $g_H(n)=1$ by applying the mex rule as $g_{F_{i,k,\ell}}(n)\ne g_{F_{i,k,\ell}}(i)=1$.

Suppose the theorem holds for all $a_n<m$:

If $a_n=m$, we want a game $H$ such as $g_H(n')=a_{n'}$ for all $n'\le n$. We take the game $H=\{G_{i},F_{k,\ell}|0\le i <m, 0\le k <n, 0\le \ell<a_k\}$ with $g_{G_{i}}(n)=i$ and $g_{G_{i}}(j)>a_j$ for $j<n$ and $F_{k,\ell}$ such that $g_{F_{k,\ell}}(n)<m$, $g_{F_{k,\ell}}(k)=\ell$ and $g_{F_{k,\ell}}(j)>a_j$ for all $j<n$ with $j\ne k$. Such $G_i$ and $F_{k,\ell}$ exist by induction hypothesis.
For this game we have $g_H(n')=a_{n'}$ for all $n'< n$ and $g_H(n)=m$ by applying the mex rule.

This gives us the induction and proves the first part of the theorem.\\

We now prove that the theorem works for the extended split function.

By Proposition~\ref{proposition:bijection}, we know that the extended Grundy function is the inverse of the extended split function. Then by the first part of the theorem, we can take a game $G$ such that $g_G(a_i)=i$ for all $0\le i \le n$ giving us that $s_G(i)=a_i$ for all $0\le i \le n$.
 \end{proof}


\subsubsection{Acknowledgement} 
The author has no competing interests to declare that are relevant to the content of this article. I would like to extend my gratitude to Doctor Koki Suetsugu for welcoming me, introducing me to combinatorial game theory, having stimulating discussion and giving me valuable advice and a lot of warm helps, especially with proofreading and correcting errors.  
I also give my sincere gratitude to my supervisor Professor Takeaki Uno for accepting me for these internship and for all the help in the administrative aspects.

\printbibliography

\end{document}